\documentclass[a4paper,11pt]{article}
\usepackage[top=2.5cm,bottom=2.5cm,left=2.2cm,right=2.2cm]{geometry}
\usepackage{amsfonts}
\usepackage{mathrsfs,amscd,amssymb,amsthm,amsmath,bm,graphicx,psfrag,subfigure,url,mathtools}
\usepackage{pict2e}
\usepackage{stfloats}
\usepackage{psfrag,amsmath}
\usepackage{tikz}
\usepackage{indentfirst}
\usepackage{hyperref}
\usepackage{bookmark}
\usepackage{enumerate}
\usepackage{enumitem}
\usepackage{latexsym,euscript,epic,eepic,color}
\usepackage{multirow}
\usepackage{multicol}
\usepackage{longtable}
\usepackage{adjustbox}
\usepackage[all]{xy}
\usepackage{setspace}
\usepackage{epstopdf}
\allowdisplaybreaks
\usepackage{authblk}
\usepackage{cite}
\makeatletter

\usepackage{authblk}

\renewcommand{\@seccntformat}[1]{{\csname the#1\endcsname}{\normalsize .}\hspace{.5em}}
\makeatother

\usepackage{ifpdf}
\usepackage{indentfirst}

\newtheorem{theorem}{Theorem}[section]
\newtheorem{proposition}[theorem]{Proposition}

\newtheorem{claim}{Claim}
\newtheorem{lemma}[theorem]{Lemma}

\newtheorem{conjecture}{Conjecture}[section]

\newcommand{\ex}{\operatorname{ex}}

\newenvironment{wst}
{\setlength{\leftmargini}{1.5\parindent}
\begin{itemize}
\setlength{\itemsep}{-1.1mm}}
{\end{itemize}}

\title{\bf  
On the large-clique version of the  Erd\H{o}s-S\'os theorem\thanks{
{\it Email addresses}: chengkunmath@163.com (K. Cheng), tyr2290@163.com (Y. Tang).}}

\author{Kun Cheng}
\author{Yurui Tang\thanks{Corresponding author}}
\affil{Department of Mathematics, East China Normal University, Shanghai, 200241, China}
\date{\today}
\begin{document}
\maketitle

\begin{abstract}
For graphs $H$ and $F$, let $\operatorname{ex}(n,H,F)$ denote the maximum number of copies of $H$ in an $F$-free graph of order $n$. Motivated by the  Erd\H{o}s-S\'{o}s theorem, Gerbner and Palmer and, independently, Zhao and Peng  conjectured that for every tree $T$ of order $k$ and every $3\le r\le k-1,$  
$$
\operatorname{ex}(n,K_r,T)
=
a\binom{k-1}{r}+\binom{b}{r},
$$
where $n=a(k-1)+b$ with $0\le b<k-1$. In this paper, we confirm the conjecture for
$r\ge \left\lceil (2k-1)/3\right\rceil$ and characterize all extremal graphs.
\vspace{0.2cm}

\end{abstract}

\section{\normalsize Introduction}
\noindent We consider finite simple graphs and use standard terminology and notation from \cite{Bollobas}, \cite{bondy1} and \cite{West}.    
Let $G$ be a graph with vertex set $V(G)$ and edge set $E(G)$. 
Then $n(G):=|V(G)|$ and $e(G):=|E(G)|$ are called the {\it order} and {\it size} of $G$, respectively.  
For $v\in V(G)$, let $N_G(v)$ be the neighborhood  of $v$ in $G$. 
We write $K_n$ for  the complete graph of order $n$. 
Two graphs are {\it disjoint} if they have no vertex in common. The {\it union} of simple graphs $G$ and $H$ is the graph $G\cup H$ with vertex set $V (G)\cup V (H)$ and edge set $E (G)\cup E (H).$ If $G$ and $H$ are disjoint, we refer to their union as a {\it disjoint union,} and generally denote it by $G + H.$  
For graphs we will use equality up to isomorphism, so $G = H$ means that $G$ and $H$ are isomorphic.
A {\it tree} is a connected graph without cycles, and a {\it forest} is a graph without cycles.

For graphs $H$ and $F$, let $\mathcal{N}(H,G)$ denote the number of copies of $H$ in a graph
$G$, and let
$$
\ex(n,H,F)=\max\{\mathcal{N}(H,G): |G|=n \text{ and } G \text{ is }F\text{-free}\}.
$$
An $n$-vertex $F$-free graph attaining this maximum is called an
\emph{extremal graph}.
Thus the classical Tur\'an number $\ex(n,F)$ is the case $H=K_2.$ 
For each non-bipartite graph $F,$ 
the celebrated Erd\H{o}s-Stone-Simonovits theorem~\cite{Erdos2,Erdos3} determines an asymptotic formula for $\ex(n, F).$  Hence it is natural and interesting to study the Tur\'an number $\ex(n,B)$ of a bipartite graph $B.$ For more advances in this topic, we refer the reader to the survey~\cite{Furedi}.  The Erd\H{o}s-S\'os conjecture~\cite{Erdos1}, 
recently proved (see~\cite[Appendix B.4]{Ad1}, \cite{Wood,wuhehui}),
can be stated as follows.

\begin{theorem}
If $G$ is a graph of order $n$ and size $e(G)> n(k-2)/2,$ then $G$ contains every tree of order $k.$  
\end{theorem}

The systematic study of this
generalized Tur\'an function was initiated by Alon and Shikhelman~\cite{AlonShikhelman};
see papers~\cite{chenyang,Gerbner,luo,wang,wuyuan,ZnPeng,zhuchen} and a nice survey~\cite{GerbnerPalmerSurvey}.

Let $T$ be a tree of order $k,$  and let $3\le r\le k-1$. If
$n=a(k-1)+b,$  where $0\le b<k-1,$ 
then the graph
$aK_{k-1}+ K_b$ 
is $T$-free, and 
$$
\mathcal{N}(K_r, aK_{k-1}+ K_b)= a\binom{k-1}{r}+\binom br.
$$
Gerbner and Palmer~\cite[Conjecture 6.5]{GerbnerPalmerSurvey} 
and, independently,  
Zhao and Peng~\cite{Peng1}
conjectured that this construction is always
extremal.

\begin{conjecture}[Gerbner-Palmer~\cite{GerbnerPalmerSurvey}, Zhao-Peng~\cite{Peng1}]\label{GP}
Let $T$ be a tree of order $k,$ let $3\le r\le k-1$ and let 
$n=a(k-1)+b$ with $0\le b<k-1.$   Then
$$
\ex(n,K_r,T)=a\binom{k-1}{r}+\binom br.
$$
\end{conjecture}

Zhou and Yuan~\cite{ZhouYuan} verified Conjecture~\ref{GP} for
$r=k-2$ and for $r=k-3\ge5$, and in those cases they also characterized the extremal graphs. 
Zhao and Peng~\cite{Peng1} confirmed Conjecture~\ref{GP} for $r=k-d$, where $d\ge 2$ and $k\ge d^2-d+3$, and characterized all extremal graphs. 
In this paper, we show that Conjecture~\ref{GP} holds for a range whose width grows linearly with $k$, and characterize all extremal graphs. Our main result is as follows. 

\begin{theorem}\label{main}
Conjecture~\ref{GP} holds for $r \ge \left\lceil (2k-1)/3\right\rceil.$ 
Moreover, the extremal graphs are exactly the following:
\begin{wst}
	\item[\rm (i)] If $0\le b<r,$ then $G=aK_{k-1}+H,$ where $H$ is an arbitrary graph of order $b;$
	\item[\rm (ii)] If $r\le b<k-1,$ then $G=aK_{k-1}+K_b.$ 
\end{wst}
\end{theorem}

\section{\normalsize Preliminaries}
A graph $G$ is called a
\emph{split graph} if there exists a partition $(V_1,V_2)$ of $V(G)$ such that
$V_1$ is a clique and $V_2$ is an independent set. In addition, if every vertex
of $V_1$ is adjacent to every vertex of $V_2$, then $G$ is called a
\emph{complete split graph}.

We first pose a lemma about tree partition.
\begin{lemma}\label{tree}
Let $k$ and $r$ be two positive integers with $r\ge \left\lceil (2k-1)/3\right\rceil,$
and let $T$ be a tree of order $k.$ Then $T$ contains a vertex $u$ such that the components of $T-u$ can be divided into
two families that respectively contain at most $r-1$ vertices.
\end{lemma}

\begin{proof} 
If $r\ge k,$ the result holds trivially. 
We next consider the case $r\le k-1.$  We show that there exists a vertex $u$ of $T$ such that each component of $T-u$  has order at most $k/2$. For each $w\in V(T)$, let
$$
m(w)=\max\{|T'|: T' \text{ is a component of }T-w\}.
$$
Choose $u\in V(T)$ such that $m(u)$ is minimum. Suppose to the contrary that $m(u)>k/2.$ 
Let $T^*$ be the component of $T-u$ with $|T^*|=m(u).$ Let $v$ be the unique neighbor of $u$
in $T^*.$ In $T-v$, the component containing $u$ has order $k-m(u)<m(u),$ while every
other component is contained in $T^*-v$ and hence has order at most $m(u)-1.$ Thus
$m(v)<m(u),$ which contradicts the choice of $u.$ Therefore, $m(u)\le k/2.$

Next we are to show that 
the components of $T-u$ can be divided into
two families that respectively contain at most $r-1$ vertices. 
Let $C_1,\dots,C_m$ be the components of $T-u$ with $a_i=|C_i|$, where $a_1\le\cdots\le a_m.$ Then
\begin{align}\label{ai}
	a_i\le \frac{k}{2}\le r-1
	~\text{for each }1\le i\le m,
	~\text{ and }
	\sum_{i=1}^m a_i=k-1.
\end{align}

If $a_m\ge k-r$, let the first family consist of that component $C_m$ and let the second
family consist of all remaining components. By \eqref{ai}, the first family has order at most $r-1$, whereas the second has order at most $k-1-(k-r)=r-1.$

We may therefore assume that $a_m\le k-r-1.$ Hence $a_i\le k-r-1$ for each $1\le i\le m$. Add components to the first
family until their total order is at least $k-r$. 
Since $a_i\le k-r-1$ and $\left\lceil (2k-1)/3\right\rceil\le r,$ it follows that the first family has order at most
$2(k-r-1)\le r-1$. Similarly, the second family consists of all remaining components and hence  has order at most $k-1-(k-r)=r-1.$ 
This proves Lemma~\ref{tree}. 
\end{proof}
Using Lemma~\ref{tree}, we get the following.
\begin{proposition}\label{tree partition}
Let $r,k$ be two positive integers with $r\ge \lceil(2k-1)/3\rceil$, and let $T$ be a tree of order $k.$ Let $G_1$ and $G_2$ be two complete graphs of order $r.$ If $V(G_1)\cap V(G_2)\neq \emptyset$ and $|V(G_1)\cup V(G_2)|\ge k,$ then  $G_1\cup G_2$ contains the tree $T.$
\end{proposition} 
\begin{proof} 
For the tree $T,$ by Lemma~\ref{tree}, there exists a vertex $u\in V(T)$ 
such that $T-u$ contains two families of components with order $s$ and $k-1-s,$ respectively, where $\max\{s,k-1-s\}\le r-1.$  Since $V(G_1)\cap V(G_2)\neq \emptyset,$ there exists a vertex $x\in V(G_1)\cap V(G_2)$ corresponding to the vertex $u\in V(T).$ 

If $|V(G_1)\setminus V(G_2)|\ge s,$ 
then $G_1- V(G_2)$ contains the subgraph of $T$ induced by the first family, and $G_2-x$  contains the subgraph of $T$ induced by the second family, because $|G_2-x|=r-1.$ Thus $G_1\cup G_2$ contains $T.$ 

If $|V(G_1)\setminus V(G_2)|< s,$ then we may choose a vertex subset $U$ with $V(G_1)\setminus V(G_2)\subseteq U\subseteq V(G_1)$ and $|U|=s$ such that  the subgraph of $G_1$ induced by $U$ contains the subgraph of $T$ induced by the first family. 
Then $ (V(G_1)\cup V(G_2) ) \setminus \left( \{x\}\cup U \right)= V(G_2)  \setminus \left( \{x\}\cup U \right)$ has order at least $k-1-s,$ and thus contains the subgraph of $T$ induced by the second family. 
This proves Proposition~\ref{tree partition}. 
\end{proof}

\begin{lemma}\label{forest}
Let $c\ge 1$ be an integer and let $F$ be a forest of order at least $2c.$ Then $F$ contains an
independent set $S$ such that $|S|=c$ and $|N_F(S)|\le c.$
\end{lemma}

\begin{proof}
We use induction on $c$. If $c=1$, the assertion is trivial. Assume $c\ge2$ and the result holds for $c-1$. 
If $F$ has an isolated vertex $u$, then $F-u$ has at least $2(c-1)$ vertices. By the
induction hypothesis, $F-u$ contains an independent set $S'$ such that $|S'|=c-1$ and
$|N_{F-u}(S')|\le c-1$. Set $S=S'\cup\{u\}$. Then $|S|=c$ and $|N_F(S)|\le c.$

We next assume that $F$ has no isolated vertices. Choose a leaf $u\in F$ and let
$v$ be its unique neighbour. The forest $F'=F-\{u,v\}$ has at least $2(c-1)$ vertices. By the
induction hypothesis, $F'$ has an independent set $S'$ of
order $c-1$ satisfying $|N_{F'}(S')|\le c-1$. Since $N_F(u)=\{v\},$ the
set $S=S'\cup\{u\}$ is independent and
$N_F(S)\subseteq N_{F'}(S')\cup\{v\}.$
Hence $|N_F(S)|\le c$. 
This proves Lemma~\ref{forest}. 
\end{proof}
By Lemma~\ref{forest}, we obtain the following.
\begin{proposition}\label{forest partition}
Let $k,c$ be two positive integers with $k\ge 2c,$ and let $F$ be a forest of order $k.$  Let $G_1$ be a complete graph of order at least $k-c,$ 
and let $G_2$ be a complete split graph of order $2c$ with $V(G_2)=R\cup Y$, where $R$ is a $c$-clique and $Y$ is an independent set of size $c.$ If $V(G_1)\cap V(G_2)=R,$ then  $G_1\cup G_2$ contains the forest $F.$ 
\end{proposition}

\begin{proof}
From Lemma~\ref{forest}, there exists an independent set $S$ in the forest 
$F$ such that $|S|=c$ and $|N_F(S)|\le c.$ 
Then the subgraph of $G_1\cup G_2$ induced by $Y$ contains the subgraph of $F$ induced by $S.$  
Since $|N_F(S)|\le c=|R|,$ 
then the subgraph of $G_1\cup G_2$ induced by $V(G_1)$ contains the subgraph of $F$ induced by $V(F)\setminus S.$  Note that $G_1$ is a complete graph and $G_2$ is a complete split graph. Therefore,  $G_1\cup G_2$ contains the forest $F.$ This completes the proof of Proposition~\ref{forest partition}.
\end{proof}

The following lemma will be needed. 

\begin{lemma}[Zhao-Peng~\cite{Peng1}, Zhou-Yuan~\cite{ZhouYuan}]\label{ZY}
Let $k-1\ge r\ge 2,$ let $a\ge0$, and let $0\le b<k-1.$ Assume that 
$x_1,\dots,x_m$ are integers with $r\le x_i\le k-1$ and
$\sum_{i=1}^m x_i\le a(k-1)+b.$ Then
$$\sum_{i=1}^m \binom{x_i}{r}
\le a\binom {k-1}r+\binom br.$$  
Equality holds only if the following conditions hold:
\begin{wst}
	\item[\rm (i)] if $b<r$, then $m=a$ and $x_1=\cdots=x_a=k-1;$
	\item[\rm (ii)] if $b\ge r$, then the multiset $\{x_1,\dots,x_m\}=\{\underbrace{k-1,\ldots,k-1}_{a},b\}.$
\end{wst}
\end{lemma}

The following lemma is an immediate consequence of Vandermonde's identity.
\begin{lemma}\label{vandermonde}
For nonnegative integers $m,n$ and $s$ with $\min\{m,n\}\ge s,$  
$$
\sum_{j=0}^{s}\binom{m}{j}\binom{n}{s-j}
=
\binom{m+n}{s}.
$$
\end{lemma}

\section{\normalsize Proofs} 
The following notations are frequently used in the forthcoming discussions. 
For a graph $G,$ let $\mathcal{K}_r(G)$ denote the family of 
$r$-cliques of $G$. 
Define an auxiliary graph 
$\Gamma_r(G)$ 
with vertex set $\mathcal{K}_r(G),$ 
where two members are adjacent if they intersect. 
Then $V (\mathcal{C})$ is the vertex set of a component $\mathcal{C}$ of $\Gamma_r(G).$  Hence, $V (\mathcal{C})\subseteq \mathcal{K}_r(G).$ 
The components of $\Gamma_r(G)$ are called the {\it $K_r$-components} of $G.$ 
For a $K_r$-component $\mathcal{C}$ of $G,$  
let 
$U(\mathcal{C})=\bigcup_{A\in V(\mathcal{C})}A.$ Hence, $U(\mathcal{C})\subseteq V(G),$ 
and the sets $U(\mathcal{C})$ corresponding to distinct $K_r$-components are vertex-disjoint.

\begin{lemma}\label{intersect}
Let $r$ and $k$ be two integers with $k-1\ge r\ge \left\lceil (2k-1)/3\right\rceil$.
Let $T$ be a tree of order $k$, and let $G$ be a $T$-free graph. 
\begin{wst}
	\item[{\rm (i)}] For any pair $r$-cliques $A$ and $B$ in $G$, if $A\cap B\ne \emptyset$, then $|A\cup B|\le k-1$. Moreover, $|A\cap B|\ge 2r-k+1$.
	\item[{\rm (ii)}] Any two members of the same $K_r$-component of $G$ intersect at least $2r-k+1$ vertices.
\end{wst}
\end{lemma}
\begin{proof}
(i) Note that $|A|=|B|=r$ and $|A\cap B|\ge 1$. If $|A\cup B|\ge k$, then by Proposition~\ref{tree partition}, $G[A\cup B]$ contains a subgraph isomorphic to $T$, a contradiction.
Therefore $|A\cup B|\le k-1,$
and hence $|A\cap B|=|A|+|B|-|A\cup B|\ge 2r-k+1.$

(ii) 
Note that a graph is connected if and only if there is a path between any two vertices. 
Let $A_0A_1\cdots A_t$ be a path in a $K_r$-component, where
$A_0,A_1,\dots,A_t$ are $r$-cliques of $G.$ It suffices to prove that
$|A_i\cap A_j|\ge 2r-k+1$ for $i,j\in\{0,1,2,\dots,t\}.$

We assert that $|A_0\cap A_i|\ge 2r-k+1$. Use induction on $i$. If $i=1,$ then by (i), $|A_0\cap A_1|\ge 2r-k+1.$ Suppose it holds for
$i-1$. Thus, $|A_0\cap A_{i-1}|\ge 2r-k+1.$ By (i), $|A_{i-1}\cap A_i|\ge 2r-k+1.$ 
Hence
$$|A_0\cap A_i|\ge 2(2r-k+1)-r=3r-2k+2\ge1,$$
where the last inequality is exactly $r\ge \left\lceil (2k-1)/3\right\rceil$. Thus $A_0$ and $A_i$ intersect. By (i), it follows that $|A_0\cap A_i|\ge 2r-k+1.$
\end{proof}

We are now ready to prove Theorem~\ref{main}.
\begin{proof}[\bf Proof of Theorem~\ref{main}]
Let $T$ be a tree of order $k.$ The graph $aK_{k-1}+ K_b$ is $T$-free and hence
$$
\ex(n,K_r,T)\ge \mathcal N(K_r,aK_{k-1}+ K_b)=  a\binom {k-1}r+\binom br.
$$
Let $G$ be an arbitrary $T$-free graph of order $n$. We are to show that  $\mathcal N(K_r,G) \le a\binom {k-1}r+\binom br.$ 
Let 
$\mathcal C_1,\dots,\mathcal C_m$
be all components of $\Gamma_r(G),$ 
that is, $\mathcal C_1,\dots,\mathcal C_m$
are all $K_r$-components of $G$. 
\begin{claim}\label{c1}
	$|\mathcal{C}_i|\le \binom {k-1}{r}$
	for $1\le i\le m,$
	with equality if and only if
	$G[U(\mathcal C_i)]= K_{k-1}.$ 
\end{claim}
\begin{proof}[Proof of Claim~\ref{c1}]
	Choose an arbitrary $K_r$-component $\mathcal C$ of $G$. 
	Choose an $r$-clique $A\in\mathcal C$. Let 
	$X=U(\mathcal C)\setminus A$ and let $t=2r-k+1.$ Recall that $r\le k-1.$ Then $t\le r.$
	Let $c=k-r.$ 
	Since $r\ge \left\lceil (2k-1)/3\right\rceil$, it follows that $2r-k+1\ge k-r$, i.e., $t\ge c.$
	For any subset $Q\subseteq A$, write 
	$\widetilde N_X(Q)=X\cap\bigcap_{v\in Q}N_G(v).$ 
	We assert that
	\begin{equation}\label{neighbor}
		\left| \widetilde N_X(Q) \right| \le c-1,~\text{for each $Q\subseteq A$ with}~|Q|\ge t.
	\end{equation}
	Suppose to the contrary that $|\widetilde N_X(Q)|\ge c$ for some $Q\subseteq A$ with $|Q|\ge t.$ Choose a set 
	$Y\subseteq \widetilde N_X(Q)$ with $|Y|=c$. 
	Since $c\le t$, we can choose a subset $R\subseteq Q$ with $|R|=c.$ Then $R$ is a clique. Since $Y\subseteq \widetilde N_X(Q),$ each vertex of $Y$ is adjacent to $R.$
	Note $|A|=r=k-c$. 
	Since $r\ge \left\lceil (2k-1)/3\right\rceil$,  we have $k\ge2(k-r)=2c.$  
	Then by Proposition \ref{forest partition}, $G[A\cup Y]$ contains the tree $T,$ a contradiction. Thus, \eqref{neighbor} holds.
	
	For $0\le j\le c-1$, write 
	$\mathcal B_j:=\{B\in V(\mathcal C):|B\setminus A|=j\}.$ Next we determine $|\mathcal B_j|$.
	By Lemma~\ref{intersect}(ii), every $r$-clique $B\in V(\mathcal C)$ satisfies
	$|A\cap B|\ge t$. Choose a $B\in\mathcal B_j$ and let $Q:=A\cap B$. Then 
	$r-j=|Q|\ge t$ and $B\setminus A\subseteq \widetilde N_X(Q).$ 
	By \eqref{neighbor}, $|\widetilde N_X(Q)|\le c-1$. For each fixed
	$Q\subseteq A$, there are at most $\binom{c-1}{j}$ choices for
	$B\setminus A$, while there are $\binom r{r-j}=\binom rj$ choices for $Q$. Hence 
	\begin{align}\label{ey2}
		|\mathcal B_j|\le \binom rj\binom{c-1}{j}~\text{for}~0\le j\le c-1.
	\end{align}
	Summing over $j$ and applying Lemma~\ref{vandermonde} gives 
	\begin{align*}
		|\mathcal C|
		\le \sum_{j=0}^{c-1}\binom rj\binom{c-1}{j}
		=\binom{r+c-1}{c-1}=\binom {k-1}{k-1-r}
		=\binom {k-1}r,
	\end{align*}
	which proves the first part of Claim~\ref{c1}. 
	
	Next we consider the equality case in  Claim~\ref{c1}. Then every inequality above becomes equality. 
	The sufficiency is immediate. We now prove the necessity.
	
	If $c=1,$ then by \eqref{neighbor}, $V(\mathcal C)=\{A\}.$ Thus 
	$G[U(\mathcal C)]= K_{k-1},$ and the equality case in Claim~\ref{c1} holds.   		
	Next we assume that $c\ge 2.$ 
	Let $A=\{a_1,a_2, \ldots,a_r\}.$ 
	We first assert that  
	\begin{align}\label{ey7}
		\widetilde N_X(A\setminus\{a_1\})=\widetilde N_X(A\setminus\{a_2\})=\cdots=\widetilde N_X(A\setminus\{a_r\}).
	\end{align}
	In fact, since inequality \eqref{ey2} becomes equality, for $j=1,$ 
	we have 
	$|\mathcal B_1|= r(c-1),$ and thus 
	\begin{align}\label{ey4}
		\text{
			$\left| \widetilde N_X(A\setminus\{a_i\}) \right|= c-1,$ for each $a_i\in A.$ 
		}
	\end{align} 			
	If $c=2,$ then \eqref{ey4} gives  
	$\left|\widetilde N_X(A\setminus\{a_i\})\right|= 1$ for each $a_i\in  A.$ 
	Suppose to the contrary that 
	$\widetilde N_X(A\setminus\{a_i\})\neq \widetilde N_X(A\setminus\{a_j\})$ 
	for some $i\neq j.$ 
	Then the two cliques 
	$(A\setminus\{a_i\})\cup \widetilde N_X(A\setminus\{a_i\})$ 
	and 
	$(A\setminus\{a_j\})\cup \widetilde N_X(A\setminus\{a_j\})$ 
	have intersection of order $r-2,$ 
	which contradicts Lemma~\ref{intersect}(ii). 
	Then \eqref{ey7} holds for $c=2.$   
	
	If $c\ge 3,$ 
	then for any two vertices $a_i,a_j\in A,$ 
	by definition, 
	we have 
	\begin{align}\label{ey5}
		\widetilde N_X\left( A\setminus\{a_i\}\right) \cup \widetilde N_X\left( A\setminus\{a_j\}\right) \subseteq \widetilde N_X\left( A\setminus\{a_i,a_j\}\right). 
	\end{align}
	Since 
	$|A\setminus \{a_i,a_j\}|=r-2\ge t,$  \eqref{neighbor} means that 
	$|N_X\left( A\setminus\{a_i,a_j\}\right)|\le c-1.$ 
	This, together with \eqref{ey4} and \eqref{ey5}, implies that 
	\begin{align*}
		\widetilde N_X\left( A\setminus\{a_i\}\right)= \widetilde N_X\left( A\setminus\{a_j\}\right) ~\text{for any two vertices}~a_i,a_j\in A, 
	\end{align*}
	which completes the proof of \eqref{ey7}.

	For a clique $B\in V(\mathcal C)\setminus\{A\},$ 
	we choose a vertex $a_k\in A\setminus B.$ 
	By definition, 
	\begin{align}\label{ey9}
		\widetilde N_X\left( A\setminus\{a_k\}\right) \subseteq \widetilde N_X\left( A\cap B\right). 
	\end{align}
	We again apply Lemma~\ref{intersect}(ii) to obtain 
	$|A\cap B|\ge t.$ With \eqref{neighbor} this gives 
	$|\widetilde N_X\left( A\cap B\right)|\le c-1.$ This, together with 
	\eqref{ey9} and \eqref{ey4}, implies that 
	\begin{align*}
		\text{
			$\widetilde N_X\left( A\setminus\{a_k\}\right) = \widetilde N_X\left( A\cap B\right),$ 
			and thus  
			$B\setminus A \subseteq \widetilde N_X\left( A\cap B\right)=\widetilde N_X\left( A\setminus\{a_k\}\right).$
		}
	\end{align*}
	By the arbitrariness of $B$ and \eqref{ey7}, 
	we have that 
	$U(\mathcal C)=A\cup \widetilde N_X\left( A\setminus\{a_1\}\right),$
	which, together with \eqref{ey4} yields 
	$|U(\mathcal C)|=r+c-1=k-1.$ 
	Note that the number of $r$-cliques in $G[U(\mathcal C)]$ is $\binom {k-1}{r}.$ 
	Then $G[U(\mathcal C)]=K_{k-1},$ 
	and the proof of Claim~\ref{c1} is complete.
\end{proof}

Now we come back to show our result. 
By definition, $|U(\mathcal C_i)|\ge r$ and $U(\mathcal C_i)\cap U(\mathcal C_j)=\emptyset$ for any distinct $i,j.$ Thus
$\sum_{i=1}^m |U(\mathcal C_i)|\le n.$  
We write
$$
I=\{i:r\le|U(\mathcal C_i)|\le k-1\}
~~\text{and}~~
J=\{i:|U(\mathcal C_i)|\ge k\}.
$$
If $i\in I$, then trivially
$|\mathcal C_i|\le\binom{|U(\mathcal C_i)|}{r},$ 
and if $i\in J$, then by Claim~\ref{c1},
$|\mathcal C_i|\le\binom {k-1}r.$ 
Consequently,
\begin{equation}\label{count}
	\mathcal N(K_r,G)
	\le
	\sum_{i\in I}\binom{|U(\mathcal C_i)|}{r}
	+|J|\binom {k-1}r.
\end{equation} 
Note that
$$
\sum_{i\in I}|U(\mathcal C_i)|+|J|(k-1)
\le
\sum_{i\in I}|U(\mathcal C_i)|+\sum_{i\in J}|U(\mathcal C_i)|
\le n=a(k-1)+b.
$$
Hence, 
\begin{align}\label{eq2}
	\sum_{i\in I}|U(\mathcal C_i)|\le (a-|J|)(k-1)+b.
\end{align}
Since $|U(\mathcal C_i)|\le k-1$ for each $i\in I,$ by \eqref{count} and Lemma~\ref{ZY}, it follows that
$$
\mathcal N(K_r,G)\le \sum_{i\in I}\binom{|U(\mathcal C_i)|}{r}
+|J|\binom {k-1}r\le (a-|J|)\binom {k-1}r+\binom br+|J|\binom {k-1}r= a\binom {k-1}r+\binom br,
$$
as desired.  

Now we characterize the extremal graphs. 
Assume that 
$G$ is an extremal graph. 
Then every inequality above becomes equality.  
Thus 
\begin{align}\label{eq1}
	\text{for each $i\in I,$ we have 
		$|\mathcal C_i|=\binom{|U(\mathcal C_i)|}{r},$
	}
\end{align}			 
and  for each $j\in J$ (if any) we have 
$|\mathcal C_j|=\binom {k-1}r.$ 
Suppose to the contrary that there exists $j\in J.$  
Then by Claim~\ref{c1}, $G[U(\mathcal C_j)]= K_{k-1},$ 
which contradicts the definition of $J.$ 
This implies that $J=\emptyset.$  
Then from \eqref{eq1}, we have 
\begin{align}\label{ep1}
	\text{
		$G[U(\mathcal C_i)]=K_{|U(\mathcal C_i)|}$ 
		for each $1\le i\le m.$
	}
\end{align}

If $0\le b<r,$ 
then from \eqref{eq2}, $J=\emptyset$ and Lemma~\ref{ZY}(i), we have that 
\begin{align*}
	\text{
		$m=a$~~and~~$|U(\mathcal C_1)|=\cdots=|U(\mathcal C_a)|=k-1.$ 
	}
\end{align*}
Since  $G$ is $T$-free, for each $1\le i\le a,$ there exists no edge with one end in $U(\mathcal C_i)$ and the other in $V(G)\setminus U(\mathcal C_i).$
This together with \eqref{ep1}  implies that 
$G=aK_{k-1}+H,$ where $H$ is an arbitrary graph of order $b,$ 
as desired.

If  $r\le b< k-1,$ then by Lemma~\ref{ZY}(ii), 
we have that 
$$\{|U(\mathcal C_1)|,\dots,|U(\mathcal C_m)|\}=\{\underbrace{k-1,\ldots,k-1}_{a},b\}.$$
This, combined with the fact that $G$ is $T$-free and \eqref{ep1}, implies that 
$G=aK_{k-1}+K_b.$ 
The proof of Theorem~\ref{main} is complete.
\end{proof}

\section*{\normalsize Declaration}

\noindent\textbf{AI Use}
During the preparation of this manuscript, the authors used ChatGPT 5.6 to assist with language editing, grammar, clarity, and formatting. The authors reviewed and verified all AI-assisted edits and take full responsibility for the content of the manuscript. 

\noindent\textbf{Conflict~of~interest}
The authors declare that they have no known competing financial interests or personal relationships that could have appeared to influence the work reported in this paper.

\noindent\textbf{Data~availability}
No data was used for the research described in the article.

\end{document}